\documentclass[a4paper,12pt]{amsart}
\usepackage[T2A]{fontenc}
\usepackage[cp1251]{inputenc}
\usepackage[english]{babel}

\usepackage{epsfig}
\usepackage{graphicx}
\usepackage{amssymb,amsfonts,amsmath}
\usepackage{longtable}

\advance\hoffset-20mm \advance\textwidth40mm

\advance\voffset-15mm \advance\textheight30mm

\newtheorem{lemma}{Lemma}
\newtheorem{theo}{Theorem}

\newtheorem{proposition}{Proposition}

\theoremstyle{definition}

\theoremstyle{remark}
\newtheorem{remark}{Remark}
\theoremstyle{remark}

\newcounter{primer}
0

\newcommand{\prr}{\par\refstepcounter{primer}%
\textbf{Example \arabic{primer}.} } 

\newcounter{sluchaj}
0

\newcommand{\slu}{\par\refstepcounter{sluchaj}%
\textbf{Case \arabic{sluchaj}.} } 

\newcounter{line}
0

\newcounter{pict}

\newcommand\kk{{\Bbbk}}

\newcommand\ZZ{{\mathbb Z}}
\newcommand\RR{{\mathbb R}}

\newcommand\QQ{{\mathbb Q}}
\newcommand\QgO{{\mathbb Q}_{\geqslant 0}}
\newcommand\ZgO{{\mathbb Z_{\geqslant 0}}}
\newcommand\vep{{\varepsilon}}
\newcommand\conv{{\rm conv}}
\newcommand\Lie{{\rm Lie}}

\def\qmatrix#1{\left(\begin{array}{*{20}r}#1\end{array}\right)}

\makeatletter
\def\kratno{\mathbin{\lower.3ex\hbox{$\m@th\vdots$}}}
\makeatother

\newcounter{itemnumber}

\begin{document}

\title[MAXIMAL TORUS ORBITS CLOSURES]{SIMPLE MODULES OF EXCEPTIONAL GROUPS\\WITH NORMAL CLOSURES OF MAXIMAL TORUS ORBITS}

\author
{Ilya~I.~Bogdanov}
\thanks{The first author is partially supported by the RFBR grant~\hbox{20-01-00096}, and by ADTP of the Ministry of Science 
and Education of Russian Federation, project \hbox{2.1.1/11133}} 
\address{Department of Higher Mathematics, Moscow Institute of Physics and Technology}
\email{ilya.i.bogdanov@gmail.com}
\author
{Karine~G.~Kuyumzhiyan}
\thanks{The second author is partially supported by <<EADS Foundation Chair in Mathematics>>,
by RFBR grant~\hbox{09-01-00648a}, by Russian-French Poncelet laboratory (UMI 2615
CNRS), and by Dmitry Zimin fund <<Dynasty>>}
\address{Chair of Higher Algebra, Lomonossov Moscow State University; Institut Fourier, Grenoble, France; Independent University of Moscow and Poncelet Laboratory}
\email{karina@mccme.ru}

\begin{abstract}
Let~$G$ be an exceptional simple algebraic group, and let $T$ be a maximal torus in~$G$. In this paper, for every such~$G$, we find all simple rational $G$-modules $V$ with the following property: for every 
vector $v\in V$, the closure of its $T$-orbit is a normal affine variety. For all $G$-modules without this property we present a $T$-orbit with the non-normal closure. To solve 
this problem, we use a combinatorial criterion of normality formulated in the terms of weights of a simple $G$-module. This paper continues~\cite{KK} and~\cite{Karina2}, where the same problem was solved for classical linear groups.
\medskip

\end{abstract}

\date{\today}

\maketitle


%
%

%

\section*{Introduction}

Let $G$ be a connected simple algebraic group over an algebraically closed field~$\kk$ of characteristic zero. Fix a maximal torus~$T$ in~$G$. 
Let~$V$ be a finite dimensional rational $G$-module. In this paper we study the normality of $T$-orbits closures in~$V$. The paper is devoted to the 
question of finding all simple modules~$V$ of the exceptional group~$G$ with the following property: for every vector $v\in V$, the closure of its 
orbit $\overline{Tv}$ is a normal (affine) algebraic variety. For the adjoint module of a semisimple group~$G$, this property was studied in~\cite{JM}. 
In the case $G=SL(n)$ this particular problem was also solved earlier in~\cite[Example 3.7]{Stu} and~\cite{Stumf}. In the previous papers of the second 
author~\cite{KK,Karina2} all the simple $G$-modules having this property were found for all classical groups~$G$. For the completeness, the results 
of~\cite{KK,Karina2} are included in the table in Theorem~\ref{mth}.

In this paper we study this property for exceptional simple algebraic groups; namely, for $E_6$, $E_7$, $E_8$, $G_2$, $F_4$, 
in all their simple modules. Recall that each simple module is defined uniquely up to isomorphism by its highest weight~$\lambda$. 
Here, $\lambda$ runs over the set of the dominant weights, that are, the weights having the form $a_1\pi_1+\ldots+a_r\pi_r$, where 
$\pi_1,\ldots,\pi_r$ are fundamental weights, and $a_1,\ldots,a_r$ are nonnegative integers. We enumerate fundamental weights as in~\cite[Chapter~4]{VO}.

The final result is the following theorem.

\begin{theo}\label{mth}
For the following types of simple algebraic groups and the corresponding modules, and for their dual modules, 
the closures of all maximal torus orbits are normal. In each other case, the module contains a maximal torus orbit with non-normal closure.
\begin{center}
\begin{longtable}{|c|c|c|}
\hline
Root system & Highest weight & Checked in \endhead
\hline 
$A_{n}, n\geqslant 1$& $\pi_1$ & \cite{KK}\\
\hline
$A_{n}, n\geqslant 1$& $\pi_1+\pi_{n}$ & \cite{KK,JM,Stu,Stumf}\\
\hline
$A_1$ &$3\pi_1$ & \cite{KK}\\
\hline
$A_1$ &$4\pi_1$  & \cite{KK}\\
\hline
$A_2$&$2\pi_1$ & \cite{KK}\\
\hline
$A_3$&$\pi_2$ & \cite{KK}\\
\hline
$A_4$&$\pi_2$ & \cite{KK}\\
\hline
$A_5$&$\pi_2$ & \cite{KK}\\
\hline
$A_5$&$\pi_3$ & \cite{KK}\\
\hline
$B_n, n\geqslant 2$&$\pi_1$ & \cite{Karina2}\\
\hline
$B_2$&$\pi_2$ & \cite{Karina2}\\
\hline
$B_2$ & $2\pi_2$ & \cite{Karina2}\\
\hline
$B_3$&$\pi_3$ & \cite{Karina2}\\
\hline
$B_4$&$\pi_4$ & \cite{Karina2}\\
\hline
$C_n, n\geqslant 3$&$\pi_1$ & \cite{Karina2}\\
\hline
$C_3$&$\pi_2$ & \cite{Karina2}\\
\hline
$C_4$ &$\pi_2$ & \cite{Karina2}\\
\hline
$D_n, n\geqslant 4$&$\pi_1$ &  \cite{Karina2}\\
\hline
$D_4$ & $\pi_2$ & \cite{Karina2}\\
\hline
$D_4$&$\pi_3$ & \cite{Karina2}\\
\hline
$D_4$&$\pi_4$ & \cite{Karina2}\\
\hline
$D_5$&$\pi_4$ & \cite{Karina2}\\
\hline
$D_6$&$\pi_5$ & \cite{Karina2}\\
\hline
$D_6$&$\pi_6$ & \cite{Karina2}\\
\hline
$F_4$ & $\pi_1$ &  \rm Case $\ref{sl1}$ \\
\hline
$G_2$ & $\pi_1$ &  \rm Case $\ref{sl2}$ \\
\hline
\end{longtable}
\end{center}
\end{theo}

In Section~\ref{predv} we present some basic notions dealing with the root systems. Also we recall a combinatorial criterion of 
simultaneous normality of all $T$-orbits closures. It allows to reduce the study of normality of $T$-orbits closures in a simple module~$V(\lambda)$ 
to the check of saturatedness of all subsets of the weight system~$M(\lambda)$.

In the next sections, we check saturatedness of all subsets in the weight systems $M(\lambda)$ of modules mentioned in Theorem~\ref{mth}. 
In every other case a non-saturated subset is indicated. To reach this goal, it is enough to treat minimal (with respect to inclusion) systems 
of weights which are not listed in Theorem~$\ref{mth}$. In each of these minimal cases, a non-saturated subset is constructed explicitly.

%

We check exceptional root systems individually, in alphabetical order. Since the construction of~$E_7$ uses the structure 
of~$E_8$, we first treat $E_8$, and then $E_7$. After that, we treat~$E_6$, using a more symmetric realization. In these cases there is 
no representation with all $T$-orbits closures being normal. Then we consider the root systems~$F_4$ and~$G_2$. In both these cases we show that 
only for the first fundamental representation all $T$-orbit closures are normal.


The authors are grateful to I.\,V.~Arzhantsev for the setting of the problem and for the fruitful discussions, and to L.~Manivel for useful explanations concerning exceptional root systems.

\medskip
%


\section{Preliminary information}
\label{predv}

Let $G$ be a connected simply connected semisimple algebraic group over an algebraically closed field~$\kk$ of characteristic~zero, let~$B$ be a 
Borel subgroup in~$G$, and let $T\subset B$ be a maximal torus. Denote by~$\Phi$ the root system of the Lie algebra $\Lie(G)$, 
corresponding to the torus~$T$. Let~$\Phi^+$ be the set of positive roots, and let $\Delta=\{\alpha_1,\ldots,\alpha_r\}$ be the simple roots 
in~$\Phi^+$ with respect to the Borel subgroup~$B$. Denote by $\pi_i$ the fundamental weight corresponding to the simple root~$\alpha_i$. 
It is well known that the weights $\pi_1,\ldots,\pi_r$ form a basis of the character lattice~$\Lambda=\Lambda(T)$ of the torus~$T$. 
In the case of a simply connected group $\Lambda(T)$ is actually the {\em weight lattice} of the root system~$\Phi$. The additive semigroup 
of~$\Lambda$ generated by fundamental weights 
coincides with the semigroup of dominant weights~$\Lambda_+$. The {\em root lattice}~$\Xi$ is the subgroup in~$\Lambda$ 
generated by the root system~$\Phi$. It is known that $\Xi\subset \Lambda$ is a sublattice of finite index, and the roots $\alpha_1,\ldots,\alpha_r$ form a basis of~$\Xi$. 
In the sequel we will treat weights $\mu\in \Lambda$ as points in the rational vector space $\Lambda_{\QQ} := \Lambda\otimes_\ZZ \QQ$. 
If the weight lattice~$\Lambda$ is realized in the space~$\QQ^n$, then we denote the standard basis in this space by~$\vep_1,\ldots,\vep_n$.

Let $W$ be the Weyl group of the root system~$\Phi$. Then $W$ can be realized as a finite group of linear transformations of~$\Lambda_\QQ$, 
generated by reflections $s_{\alpha}$, where $\alpha$ is a root, see~\cite{Hum}. Recall that the reflection $s_{\alpha}$ is given 
by the formula
$
  s_{\alpha}\colon \beta \mapsto \beta - \frac {2(\alpha, \beta)}{(\alpha, \alpha)}\alpha
$, the brackets $(\cdot,\cdot)$ denote the standard scalar product, which is always given by the identity Gram matrix in our realizations. 
Recall that for an irreducible root system~$\Phi$, the Weyl group acts transitively on the set of roots of the same length, see~\cite[Proposition 6.1.11]{Bour2}.

Let $v_1,\ldots,v_r$ be vectors in a rational vector space. For an arbitrary set of rational numbers $A$, we denote by $A(v_1,\ldots,v_r)$ 
the set of linear combinations of vectors $v_1,\ldots,v_r$ with coefficients from~$A$. The set of points $\{\mu_1, \ldots,\mu_s\}\subset\QQ^n$ 
is called {\it saturated} if
$$
\ZgO (\mu_1, \ldots,\mu_s) = \ZZ (\mu_1, \ldots,\mu_s) \cap \QgO (\mu_1, \ldots,\mu_s).
$$
The set of points $\{\mu_1, \ldots,\mu_s\}\subset\QQ^n$ is called {\it hereditarily normal} if each its subset is saturated.

Let $T$ be an algebraic torus and $\Lambda=\Lambda(T)$ be its character lattice. 
Every rational $T$-module $V$ can be decomposed in the following way:
$$
  V=\bigoplus_{\mu\in \Lambda} V_\mu, \quad\text{where}\quad
  V_\mu=\{v\in V \,|\, tv=\mu(t)v\quad \forall\, t\in T\}.
$$
We denote by $M(V)=\{ \mu\in \Lambda \,|\, V_\mu\ne 0 \}$ the set of weights of~$V$. Each nonzero vector~$v$ in $V$ has its {\it weight decomposition} 
$$
v=v_{\mu_1}+\dots+v_{\mu_s},\quad v_{\mu_i}\in V_{\mu_i},\quad v_{\mu_i}\ne 0.
$$ 
Recall that an irreducible affine algebraic variety~$X$ is called {\it normal} if its algebra of regular functions $\kk[X]$ is integrally closed in its field 
of fractions. The following statement is a well known combinatorial criterion of normality of a $T$-orbit closure, see~\cite[I, \S 1, Lemma~1]{KKMSD}.

\begin{proposition}
Let $V$ be a finite dimensional rational $T$-module and $v=v_{\mu_1}+\dots+v_{\mu_s}$ be the weight decomposition of a vector $v\in V$.
The closure $\overline{Tv}$ of the $T$-orbit of $v$ is normal if and only if $\{\mu_1,\ldots,\mu_s\}$ is saturated.
\end{proposition}

For a simple $G$-module $V(\lambda)$ with the highest weight~$\lambda$, we define by $M(\lambda)$ the set of its weights with respect to the 
maximal torus~$T$.
%
Recall the description of the set~$M(\lambda)$. The {\it weight polytope} $P(\lambda)$ of the module $V(\lambda)$ is the convex hull 
of the $W$-orbit of the point $\lambda$ in~$\Lambda_\QQ$: $P(\lambda)=\conv\{w\lambda \,|\, w\in W\}$.
Then
$$
M(\lambda)=(\lambda+\Xi)\cap P(\lambda),
$$
see~\cite[Theorem 14.18]{FH} and~\cite[Exercises to Chapter~VIII \S 7]{Bour3}.

\begin{remark}\label{notdomin}
Actually, we could start with a non-dominant weight $\lambda$ in this construction. Then the formula $M(\lambda)=(\lambda+\Xi)\cap
P(\lambda)$ gives the set of weights of a representation such that its highest weight belongs to the
$W$-orbit of the vector~$\lambda$.
\end{remark}

There is a partial order on the vector space $\Lambda_\QQ$: $\lambda \succeq \mu $ if and only if $\lambda-\mu$ is an integer linear combination of 
positive roots with nonnegative coefficients.


In the sequel we refer to a nonsaturated subset as an {\it NSS}. By an {\it extended nonsaturated subset} we mean a nonsaturated subset $\{v_1, \dots, v_r\}$
augmented by a vector $v_0$ such that 
\begin{enumerate}
\item
$v_0 \in (\ZZ (v_1, v_2, \dots,
v_r) \cap \QgO (v_1, v_2, \dots, v_r))\setminus  \ZZ_{\geqslant 0} (v_1, v_2, \dots, v_r)$,
\item 
there exists a $\QgO$-representation
$$
v_0=q_1v_{i_1}+\ldots+q_sv_{i_s}, \quad v_{i_j}\in \{v_1, v_2, \dots, v_r\}
$$
with linearly independent vectors $v_{i_1},\ldots,v_{i_s}$ and coefficients $q_i\in [0,1)$. 
\end{enumerate}

\smallskip

The fractional part of a real value~$q$ is denoted by~$\{q\}$, and the integer part by~$\lfloor q \rfloor$. The following lemmas are useful in proofs by contradiction.

\begin{lemma}[{\cite[Lemma~3]{Karina2}}]\label{lemma01}
Suppose that a set $M=\{v_1, \dots, v_r\}$ is not saturated. Then there exists a vector~$v_0$ such 
that~$\{v_0; v_1, \dots, v_r\}$ is an ENSS.
\end{lemma}



\begin{lemma}[{\cite[Lemma~1]{Karina2}}]\label{l.ovl}
Let $\lambda,\; \lambda'\in \Lambda_+$. Suppose that $\lambda \succeq \lambda'$, then $M(\lambda)\supseteq M(\lambda')$.
\end{lemma}


\begin{remark}\label{zam2}\label{ovlozhenii}
Take $\lambda'\in \Lambda_+$ and assume that $M(\lambda')$ is not hereditarily normal. By Lemma~\ref{l.ovl}, for all $\lambda \in \Lambda_+$ such 
that~$\lambda \succeq \lambda'$ the set $M(\lambda)$ is not hereditarily normal. 
\end{remark}

Let $v_0, v_1, \dots, v_r$ be vectors in a rational vector space $\QQ^n$, and let $f$ be a linear function on~$\QQ^n$. We call $f$ a {\it discriminating linear function} for the collection 
$\{v_0; v_1, \dots, v_r\}$ if the value $f(v_0)$ cannot be represented as a linear combination of the values $f(v_1),\ldots,f(v_r)$ with nonnegative integer coefficients. 
Assume that $v_0\in \ZZ (v_1, v_2, \dots, v_r) \cap \QgO (v_1, v_2, \dots, v_r)$, and that $v_0$ can be represented as a
$\QgO$-combination of linearly independent vectors $v_1, \dots, v_r$ with coefficients from the interval~$[0,1)$. Then the existence of a discriminating function guarantees 
that $\{v_0; v_1, \dots, v_r\}$ is an~ENSS.


\section{The root system $E_8$}\label{sec:e8}
Consider~$\QQ^8$ with the standard scalar product. The set of vectors
$$
\pm \vep_i \pm \vep_j \quad (1\leq i<j\leq 8)\qquad\mbox{and}\qquad \frac 12 \sum_{i=1}^8(-1)^{\nu_i}\vep_i ,\quad \mbox{where } \sum_{i=1}^8 \nu_i \mbox{ is even,}
$$
form the root system~$E_8$, see~\cite[Ch. 6, \S 4]{Bour2}.
The weight lattice~$\Lambda$ coincides with the root lattice~$\Xi$. It has the following structure:
$$
  \Lambda=\Lambda_0\cup \Lambda_1, \;
  \Lambda_0=\left\{v\in \ZZ^8: \quad \sum_{i=1}^8 v_i \mbox{ is even}\right\}, \;
  \Lambda_1=\Lambda_0+\frac12(1,1,1,1,1,1,1,1).
$$
The set $\Lambda_1$ consists of all vectors which have strictly half-integer coordinates with the even sum. We show below that 
none of the sets~$M(\lambda)$ for an arbitrary weight $\lambda\in \Lambda\setminus\{0\}$ (not only for the dominant weight, see Remark~\ref{notdomin}) is hereditarily normal. This implies that for every irreducible representation there exists a non-normal $T$-orbit closure.

Note that the Weyl group~$W$ of the root system~$E_8$ contains, in particular, all the permutations of coordinates, since they are generated by the 
reflections $s_{\alpha}$ for the vectors $\alpha$ of the form $\vep_i-\vep_j$. The group $W$ also contains sign changes in the even number 
of coordinates.
E.g. the sign change in the first and the second coordinates 
can be obtained as the superposition of reflections $s_{\alpha_1}s_{\alpha_2}$, where
$\alpha_{1,2}=\vep_1\pm\vep_2$.

Choose a weight $\lambda^\circ=\vep_1+\vep_2$. Firstly we show that $M(\lambda^\circ)$
contains an ENSS. 

\prr{$\lambda=\lambda^\circ
=\vep_1+\vep_2$.}
Consider the following vectors of~$M(\lambda^\circ)$:
\begin{gather*}
\qmatrix{v_1\\v_2\\v_3\\v_4\\v_5\\v_6\\v_7\\v_8}=\qmatrix{
1&  0&  1&  0&  0&  0& 0& 0\\
1&  0&  0&  1&  0&  0& 0& 0\\
0&  1&  0&  0&  1&  0& 0& 0\\
0&  1&  0&  0&  0&  1& 0& 0\\
0&  0& -1& -1&  0&  0& 0& 0\\
0&  0&  0&  0& -1& -1& 0& 0\\
1&  0&  0&  0&  0&  0& 1& 0\\
0& -1&  0&  0&  0&  0& 1& 0},
\end{gather*}
and the vector $v_0=(1,1,0,0,0,0,0,0)=\frac 12 (v_1+v_2+v_3+v_4+v_5+v_6)=v_7-v_8$.
Suppose that $v_0$ is represented as a $\ZgO$-combination of the other vectors.
Consider the seventh coordinate. It is clear that the coefficients at $v_7$ and $v_8$
are zero. Since the vectors $v_1,\ldots,v_6$ are linearly independent, $v_0$ has the unique decomposition as a linear combination of these six vectors. 
The coefficients of this decomposition
are non-integers. Hence the set $\{v_0;v_1,\ldots,v_8\}$ is indeed an ENSS.
\medskip

Next we show that for each nonzero weight $\lambda$, the weight $\lambda^\circ$ is contained in $M(\lambda)$.
Consider the cases $\lambda\in\Lambda_0$ and $\lambda\in\Lambda_1$ separately.

\begin{lemma}\label{ce}
If $\lambda=(a_1,\ldots,a_8)\in \Lambda_0\setminus \{0\}$, then $\lambda^\circ\in
M(\lambda)$.
\end{lemma}
\begin{proof}
Obviously, there exist two coordinates of the same parity $a_i$, $a_j$ such that
${|a_i|+|a_j|\geqslant 2}$. Without loss of generality, we may suppose that these coordinates are $a_1$ and $a_2$. Acting with the Weyl group, we can obtain
$a_1\geqslant 0$, $a_2\geqslant 0$. This gives
$a_1+a_2\geqslant 2$. Let $a=\frac{a_1+a_2}2$. Obviously, $a\geq 1$ and $a$ is an integer. Next, the set $M(\lambda)$ contains the point
$\lambda'=(a_2,a_1,-a_3,-a_4,-a_5,-a_6,-a_7,-a_8)$. Also it contains the point 
$$
\lambda''=(a,a,0,\ldots,0)
$$ 
because $\lambda''$ is the midpoint of the segment
$[\lambda, \lambda']$ and $\lambda-\lambda''\in
\Xi$. 
Obviously, the point $\lambda^\circ$ belongs to the segment $[\lambda'', -\lambda'']$, moreover, $\lambda^\circ-\lambda''\in \Xi$. Finally we obtain that $\lambda^\circ\in M(\lambda'')\subseteq M(\lambda)$.
\end{proof}

\begin{lemma}\label{nce}
If $\lambda=(a_1,\ldots,a_8)\in \Lambda_1$, then $\lambda^\circ\in M(\lambda)$.
\end{lemma}
\begin{proof}
Take the root $\alpha_0=\left(\frac 12 ,\ldots,\frac 12\right)$. Consider the reflection~$s_{\alpha_0}\in W$. Let ${a_i=b_i/2}$. Then
$b_i\in \ZZ$, the values $b_i$ are odd and such that $\sum_{i=1}^8 b_i$ is divisible by~4. Thus 
$$
s_{\alpha_0}(\lambda)=\lambda-\frac{a_1+\ldots+a_8}4(1,1,\ldots,1)=\lambda-\frac{b_1+\ldots+b_8}8(1,1,\ldots,1).
$$
If $\sum_{i=1}^8b_i$ is not divisible by~8, then the vector
$s_{\alpha_0}(\lambda)$ is nonzero and has integer coordinates. Hence
$M(\lambda)=M(s_{\alpha_0}(\lambda))\ni \lambda^\circ$ due to Lemma~\ref{ce}. 
If $\sum_{i=1}^8b_i$ is divisible by~8, choose two coordinates having the same residues modulo~4 among~$b_i$.
We may assume that these coordinates are~$b_1$ and~$b_2$. Changing their signs, we obtain the vector $\lambda'\in M(\lambda)$
such that its doubled sum of coordinates is not divisible by~8. 
Hence the vector $s_{\alpha_0}(\lambda')$ is nonzero and has integer coordinates. With the same reasoning as above we obtain
$M(\lambda)=M(s_{\alpha_0}\lambda')\ni \lambda^\circ$.
\end{proof}

We have shown that for~$E_8$ and for all~$\lambda\in \Lambda\setminus\{0\}$ the set $M(\lambda)$
contains a non-saturated subset. Therefore, any rational $E_8$-module contains a $T$-orbit with non-normal closure.

\section{The root system $E_7$}
For the root system~$E_7$, we use a realization which is slightly different from the one in~\cite[Ch.~6, \S 4]{Bour2}. Namely, we change the signs of the first and the last coordinates.
Now the set of roots belongs to the 7-dimensional subspace~$L\subset \RR^8$, 
orthogonal to the vector $\vep_7-\vep_8$. It contains the following vectors:
\begin{align*}
  &\pm \vep_i \pm \vep_j \quad (1\leq i<j\leq 6); \qquad
  \pm (\vep_7+\vep_8); \\
  &\pm \frac 12 (\vep_7+\vep_8+\sum_{i=1}^6 (-1)^{\nu_i}\vep_i), \quad \mbox{where } \sum_{i=1}^6 \nu_i \mbox{ is even.}
\end{align*}

The root lattice~$\Xi$ is the intersection of the root lattice of~$E_8$
with~$L$, namely, $\Xi=\Xi_0\cup \Xi_1$, where
$$
\Xi_0=\left\{\ell\in \ZZ^8: \quad \sum_{i=1}^6 \ell_i \mbox{ even}, \; \ell_7=\ell_8\right\}, \quad
  \Xi_1=\Xi_0+\frac12(1,1,1,1,1,1,1,1),
$$
and $\Xi_1$ again consists of all vectors in~$L$ having strictly half-integer coordinates with even sum.

The weight lattice~$\Lambda$ is a superlattice of $\Xi$ of index~2. Namely,
$$
  \Lambda=\Xi\cup\left(\Xi+\frac12(1,1,1,-1,-1,-1,0,0)\right).
$$
It consists of all vectors $\ell\in L$ such that all their coordinates are either integers or half-integers, their sum is even, and for all
$1\leq i<j\leq 6$ the value $(\ell_i-\ell_j)$ is integer.

The Weyl group~$W$ contains, in particular, all the permutations of the first six coordinates. It also contains the sign changes in the even
number of coordinates, where the two last coordinates are either both present or both absent.
The Weyl group $W$ acts on~$\Phi$ transitively.

We need analogs of Lemmas~\ref{ce} and~\ref{nce} from Section~\ref{sec:e8}. The following lemmas will also be useful in the construction of examples.

\begin{lemma}\label{ce7}
If $\lambda=(a_1,\ldots,a_8)\in \Xi\setminus \{0\}$, then $M(\lambda)$
contains all the roots.
\end{lemma}
\begin{proof}
Consider first the case $\lambda\in \Xi_0$. We claim that $M(\lambda)$ contains either the vector
$\vep_1+\vep_2$ or the vector $\vep_7+\vep_8$. Then by the transitivity of the Weyl group it contains actually all the roots. 
To prove the claim, we note that if $\sum_{i=1}^6 |a_6|\neq 0$, then this sum is even. Choose two numbers of the same parity among $a_1,\dots,a_6$ such that the sum of their absolute values 
is not less than two. Now proceed as in the proof of Lemma~\ref{ce}. We obtain that $\vep_1+\vep_2\in M(\lambda)$. 
Otherwise, if $\sum_{i=1}^6 |a_6| = 0$, we have
 $\lambda=a_7(\vep_7+\vep_8)$, and the point $\vep_7+\vep_8$ belongs to the interval $[-\lambda'',\lambda'']$. Hence it also
 belongs to~$M(\lambda)$.

Now consider the case $\lambda\in \Xi_1$. Write down
$\lambda=\frac12(b_1,\ldots,b_6,b_7,b_7)$, where all the~$b_i$s are odd. If $B=\sum_{i=1}^6 b_i +2b_7$ is not divisible by~8, then~$s_{\alpha_0}(\lambda)\in \Xi_0$, 
and this case follows from the case considered above. If~$B$ is divisible by~8, then we can choose two numbers among $b_1,\dots,b_6$ having the same residues modulo~4. As in the proof of Lemma~\ref{nce}, this allows to reduce the latter case to the preceding one.
\end{proof}

\begin{lemma}\label{nce7}
Let $\lambda=(a_1,\ldots,a_8)\in \Lambda\setminus\Xi$. Then
$M(\lambda)$ contains the vector
$$
\lambda_*=(1,0,0,0,0,0,1/2,1/2).
$$
\end{lemma}
\begin{proof}
Since $\lambda\notin\Xi$, either the six first coordinates of $\lambda$ are integers and two last are half-integers, or vice versa. 
If the last coordinates are integers,
we can apply the reflection $s_{\alpha_0}$, after a sign change in some two of the first six coordinates, if needed, in an order to obtain the vector such that its first six coordinates are integers, and the two last are half-integers. 
Now we may assume that two last coordinates are half-integers, i.e. $\lambda=(a_1,a_2,\ldots,a_6,b_7/2,b_7/2)$ for some integers
$a_1,\dots,a_6$ and an odd integer~$b_7$.

Consider the vector $\lambda'=(a_1,a_2,\ldots,a_6,-b_7/2,-b_7/2)\in M(\lambda)$.
Since $b_7$ is odd, the set $M(\lambda)$ contains the vector
$\lambda''=(a_1,a_2,\ldots,a_6,1/2,1/2)\in
[\lambda,\lambda']$. Notice that $a_1+\ldots+a_6$ is odd. Consider an odd coordinate among~$a_i$; without loss of generality we may assume that this is~$a_1$. Consider all the points which can 
be obtained from~$\lambda''$ by the sign change in an even number of coordinates from the second till the sixth. Their barycenter will be the point
$(a_1,0,\ldots,0,1/2,1/2)\in M(\lambda)$.

Finally, the point  $(1,0,\ldots,0,1/2,1/2)$ belongs to the segment joining the points 
$$(a_1,0,\ldots,0,1/2,1/2)\mbox{ and }(-a_1,0,\ldots,0,1/2,1/2).
$$ 
It follows that
$M(\lambda)$ contains the point  $\lambda_*\in M(\lambda)$.
\end{proof}

Now we are ready to present the NSSs for all the dominant weights.

\prr{The weight $\lambda\in \Xi$.}\label{e7l1}
By Lemma~\ref{ce7}, the set~$M(\lambda)$ contains all the roots. Consider the following subset of~$M(\lambda)$:
\begin{gather*}
\qmatrix{v_1\\v_2\\v_3\\v_4\\v_5\\v_6}=\qmatrix{
1&  0&  0& -1&  0&  0& 0& 0\\
0&  1&  0&  0& -1&  0& 0& 0\\
0&  0&  1&  0&  0&  1& 0& 0\\
0&  0&  0&  0&  0&  0& 1& 1\\
0&  0&  0&  1&  1&  0& 0& 0\\
1/2&  1/2&  1/2&  1/2& 1/2& 1/2& 1/2& 1/2},\\
v_0=\frac 12 (1,1,1,-1,-1,1,1,1)=\frac 12 (v_1+v_2+v_3+v_4)=v_6-v_5.
\end{gather*}
Suppose that $v_0$ is a $\ZgO$-combination of vectors $v_1,\dots,v_6$. The first three coordinates of~$v_0$ being smaller than one,
the coefficients at $v_1$, $v_2$, $v_3$ must be zero. However, in this case the fourth coordinate of this linear combination will inevitably be nonnegative, which gives a contradiction. This means that $\{v_0;v_1,\dots,v_6\}$ is indeed an ENSS.

\prr{The weight $\lambda\in\Lambda\setminus\Xi$.}\label{e7l2}
By Lemma~\ref{nce7}, the set~$M(\lambda)$ contains the vector
$$
\lambda_*=(1,0,0,0,0,0,1/2,1/2),
$$ 
hence it also contains the vector~$-s_{\alpha_0}\lambda_*=\frac12(-1,1,1,1,1,1,0,0)$. Now we can give the NSS.
$$
\qmatrix{v_1\\v_2\\v_3\\v_4\\v_5\\v_6\\
\noalign{\medskip}
v_7\\v_8\\v_9\\v_{10}\\v_{11}\\v_{12}}=\qmatrix{
-1/2 & 1/2 & 1/2 & 1/2 & 1/2 & 1/2 & 0 & 0\\
1/2 & -1/2 & 1/2 & 1/2 & 1/2 & 1/2 & 0 & 0\\
1/2 & 1/2 & -1/2 & 1/2 & 1/2 & 1/2 & 0 & 0\\
1/2 & 1/2 & 1/2 & -1/2 & 1/2 & 1/2 & 0 & 0\\
1/2 & 1/2 & 1/2 & 1/2 & -1/2 & 1/2 & 0 & 0\\
1/2 & 1/2 & 1/2 & 1/2 & 1/2 & -1/2 & 0 & 0\\
\noalign{\medskip}
0 & 0 & 0 & 1 & 0 & 0 & 1/2 & 1/2 \\
0 & 0 & 0 & 0 & 1 & 0 & 1/2 & 1/2 \\
0 & 0 & 0 & 0 & 0 & 1 & 1/2 & 1/2 \\
0 & 0 & -1 & 0 & 0 & 0 & 1/2 & 1/2 \\
0 & -1 & 0 & 0 & 0 & 0 & 1/2 & 1/2 \\
-1 & 0 & 0 & 0 & 0 & 0 & 1/2 & 1/2 \\
}.
$$
We have:
\begin{multline*}
  v_0=(1,1,1,1,1,1,0,0)
  =\frac12 (v_1+v_2+v_3+v_4+v_5+v_6)\\=v_7+v_8+v_9-v_{10}-v_{11}-v_{12}.
\end{multline*}
If we consider the 7th coordinate, we see that the $\ZgO$-combination can contain only the first 6 vectors with nonzero coefficients. Since they are
linearly independent, we cannot obtain yet another decomposition of $v_0$ in these vectors.

Thus we have constructed an NSS for each nonzero $\lambda\in \Lambda$.

\section{The root system $E_6$}
Instead of the standard realization, we will use a different realization of~$E_6$ in the 6-dimensional subspace of the 9-dimensional space because it is more symmetric. 
The coordinates of each vector are split into three triples of consecutive coordinates, and the sum in each triple equals zero.
It is easy to check that the set given below is indeed a root system. Moreover, the Gram matrix for its simple roots coincides with the Gram matrix for the
simple roots of the root system $E_6$ in its standard realization. Hence it is indeed~$E_6$.

The roots are: 18 vectors of the form $\vep_i-\vep_j$, where distinct indices $i$, $j$ belong to one of the triples 
$\{1,2,3\}$, $\{4,5,6\}$, $\{7,8,9\}$; 
27 vectors of the form $(\bar a;\; \bar b ;\; \bar c)$, and 27 vectors of the form $(-\bar a;\; -\bar b ;\; -\bar c)$, where 
$$
\bar a, \bar b, \bar c
  \in\left\{ \left(\frac 23, -\frac 13, -\frac 13 \right), \left(-\frac 13, \frac 23, -\frac 13 \right), \left(-\frac 13, -\frac 13, \frac 23\right) \right\}.
 $$\\
 
The simple roots are
\begin{gather*}
\alpha_1=(0,0,0;0,0,0;0,1,-1),\quad
\alpha_2=(0,1,-1;0,0,0;0,0,0),\\
\alpha_3=(0,0,0;0,0,0;1,-1,0),\quad
\alpha_4= \left( \frac 13, -\frac 23, \frac 13; -\frac 23, \frac 13, \frac 13; -\frac 23, \frac 13, \frac 13  \right),\\
\alpha_5=(0,0,0;1,-1,0;0,0,0),\quad
\alpha_6=(0,0,0;0,1,-1;0,0,0).
\end{gather*}
The corresponding fundamental weights are
\begin{gather*}
\pi_1= \left( \frac 23, -\frac 13, -\frac 13; 0,0,0; \frac 13, \frac 13, -\frac 23  \right),\quad
\pi_2=(1,0,-1;0,0,0;0,0,0),\\
\pi_3=\left(\frac 43, -\frac 23, -\frac 23;0,0,0;\frac 23, -\frac 13, -\frac 13\right),\quad
\pi_4= (2,-1,-1;0,0,0;0,0,0),\\
\pi_5=\left(\frac 43, -\frac 23, -\frac 23;\frac 23, -\frac 13, -\frac 13; 0,0,0\right),\quad
\pi_6=\left( \frac 23, -\frac 13, -\frac 13;  \frac 13, \frac 13, -\frac 23 ; 0,0,0 \right).\\
\end{gather*}
The Weyl group $W$ of the root system $E_6$ contains all the permutations of coordinates in triples. Moreover,
$W$ acts on the set of roots transitively.

Let us show that for all $\lambda\in \Lambda\setminus \{0\}$ the set $M(\lambda)$ is not hereditarily normal.
By Lemma~\ref{l.ovl} and Remark~\ref{zam2} it is enough to show that already in the fundamental representations not all $T$-orbit closures are normal, 
i.e., that all the sets~$M(\pi_i)$ are not hereditarily normal.

\prr \label{prr4} The cases $\lambda=\pi_1$ and $\lambda=\pi_6$ are similar, so 
we consider just the case $\lambda=\pi_1$.
Take the root $\beta_0=\frac13(1,-2,1;-2,1,1;1,1,-2)$. Then $M(\lambda)$
contains the vector
$$
 s_{\beta_0}(\lambda)=\frac13(1,1,-2;2,-1,-1;0,0,0).
$$
Hence, the set~$M(\lambda)$ contains the set
\begin{gather*}
\qmatrix{v_1\\v_2\\v_3\\v_4\\v_5\\v_6\\
\noalign{\medskip}
v_7}=\frac13\qmatrix{2&-1&-1&0&0&0&1&1&-2\\2&-1&-1&0&0&0&1&-2&1\\2&-1&-1&0&0&0&-2&1&1\\
1&1&-2&2&-1&-1&0&0&0\\1&1&-2&-1&2&-1&0&0&0\\1&1&-2&-1&-1&2&0&0&0\\
\noalign{\medskip}
-1&-1&2&0&0&0&1&1&-2}.
\end{gather*}
We have:
$$
v_0=(1,0,-1, 0,0,0,0,0,0)=1/3(v_1+v_2+v_3+v_4+v_5+v_6)=v_1-v_7.
$$
To show that $\{v_0;v_1,\dots,v_7\}$ is an ENSS, it remains to verify that
$v_0$ is not a $\ZgO$-combination of vectors $v_1,\dots,v_7$. Consider the discriminating linear function 
$$
f=45n_1+15n_2+30n_3-3n_7-3n_8-6n_9.
$$
Obviously, 15 cannot be decomposed as a sum of integers 17, 14, and 2.

\prr\label{prr5}
 $\lambda=\pi_2$. The set $M(\pi_2)$ contains all the roots. The following set is an ENSS:
\begin{gather*}
\qmatrix{v_1\\v_2\\v_3\\v_4\\v_5\\v_6\\
\noalign{\medskip}
v_7}
=\qmatrix{
1&0&-1&0&0&0&0&0&0\\
1&-1&0&0&0&0&0&0&0\\
0&0&0&1&0&-1&0&0&0\\
0&0&0&1&-1&0&0&0&0\\
0&0&0&0&0&0&1&0&-1\\
0&0&0&0&0&0&1&-1&0\\
\noalign{\medskip}
1/3&1/3&-2/3&-2/3&1/3&1/3&-2/3&1/3&1/3}, \\
\noalign{\medskip}
v_0=\frac13(2,-1,-1;2,-1,-1;2,-1,-1) =\frac13\sum_{i=1}^6 v_i=v_1-v_7.
\end{gather*}
Using the discriminating linear function $f=12n_1+3n_2$, we easily show that $v_0$ is not a $\ZgO$-combination of vectors $v_1,\dots,v_7$. 
Indeed, the values of $f$ on vectors $v_1,\dots,v_7$ equal 0, 12, 9, and~5, while $f(v_0)=7$.

\medskip

Non-saturated subsets for the other fundamental weights can be constructed with the help of Examples~\ref{prr4} and~\ref{prr5}. The weights 
$\pi_3$ and $\pi_5$ can be reduced to Example~\ref{prr4}. For instance, for
~$\pi_3$ the vector $\pi_3'=\frac13(-2,4,-2;0,0,0;2,-1,-1)$ 
belongs to~$M(\pi_3)$, so the vector 
$$
\frac13(1,1,-2;0,0,0;2,-1,-1),
$$ 
which is the midpoint of the segment~$[\pi_3,\pi_3']$, also belongs to~$M(\pi_3)$.
Up to a cyclic permutation of the triples, this midpoint is~$\pi_6$. The vector~$\pi_5$ can be reduced to~$\pi_6$ in the same way.
Finally, the weight $\pi_4$ can be reduced to~$\pi_2$, because
$$
 \pi_2=\frac13\bigl(2\cdot(2,-1,-1;0,0,0;0,0,0)+(-1,2,-1;0,0,0;0,0,0)\bigr)\in M(\pi_4).
$$
We have shown that for the root system $E_6$ the set of weights of each nonzero representation contains an NSS.


\section{The root system $F_4$}
The set of roots is the following subset of~$\QQ^4$:
\begin{gather*}
\pm \vep_i \quad(1\leqslant i\leqslant 4); \quad
\pm\vep_i\pm\vep_j \quad(1\leq i<j\leq 4); \quad
\frac 12 (\pm \vep_1\pm \vep_2 \pm \vep_3 \pm \vep_4).
\end{gather*}

The simple roots are $\frac 12 (\vep_1-\vep_2-\vep_3-\vep_4)$, $\vep_4$,
$\vep_3-\vep_4$, $\vep_2-\vep_3 $. The corresponding fundamental weights are:
\begin{gather*}
\pi_1=(1,0,0,0),\;
\pi_2=(3/2,1/2,1/2,1/2), \;
\pi_3=(2,1,1,0), \;
\pi_4=(1,1,0,0).
\end{gather*}

The root and weight lattices coincide:
$$
\Xi=\Lambda=\ZZ^4\cup \bigl(\ZZ^4+(1/2,1/2,1/2,1/2)\bigr).
$$
The Weyl group contains all the permutations of coordinates and the sign change in any subset of coordinates. 

\slu{$\lambda=\pi_1$.} 
\label{sl1}
Here $M(\lambda)$ is the set of all roots of length~1. Suppose on the contrary that there exists an ENSS $(v_0;v_1,\ldots)$.

Let us verify that all nonzero volumes of 4-tuples of vectors in~$M(\lambda)$ equal~1 or~$1/2$. Consider any vectors
$v_1,v_2,v_3,v_4$. If at least two of them are of the form~$\pm \vep_i$, then the statement holds. Indeed, if there are only two such vectors, then
$D=\det(v_1,v_2,v_3,v_4)=1\cdot(\pm 1/4 \pm 1/4)$. If there are exactly three such vectors, then $D=1\cdot (\pm 1/2)$; if all four have this form, then $D=\pm 1$. 
Otherwise at least three vectors among the $v_i$s have the form $1/2(\pm \vep_1 \pm \vep_2 \pm \vep_3 \pm \vep_4)$. At least two of these vectors
have the same parity of the sum of coordinates.
Assume that these are the vectors $v_1,v_2$, and their sums of coordinates are odd. Then apply the symmetry~$s_{\alpha_0}$ with respect to the root 
 $\alpha_0=\frac12(1,1,1,1)$ to all the four vectors. After that the vectors~$v_1$ and~$v_2$ will have the form $\pm \vep_i$, and we have already checked the statement for the new 4-tuple of vectors.

Now, if we have an ENSS $\{v_0;v_1,\dots,v_s\}$, then the corresponding 
$\QgO$-combination consists of vectors with volume~1 (otherwise this combination will be integer by the Cramer formulae). From the previous 
reasoning we may assume that these vectors are $\vep_1$, $\vep_2$, $\vep_3$, $\vep_4$. Looking at the form of the weight lattice, we see that 
$v_0=\frac 12 (\vep_1+\vep_2+\vep_3+\vep_4)$ since all the coefficients of the $\QgO$-combination are less than~1.
Then there exists a vector~$v_k$ of the form $(\pm 1/2, \pm 1/2, \pm 1/2, \pm 1/2)$ among $v_1,\ldots,v_s$. Now we can represent~$v_0$ 
as a $\ZgO$-combination of $v_1,v_2,v_3,v_4,v_k$, adding to $v_k$ 
the vectors~$\vep_i$ with the indices $i$ such that the $i$th coordinate of the vector~$v_k$ is negative. We get a contradiction.

\begin{remark}
The set $M(\lambda)$ in this case is almost unimodular, see~\cite[Section~1.2]{Karina2}. The properties of almost unimodular sets were applied before 
by the second author in the similar situations, cf.~\cite[Case~5]{Karina2}.
\end{remark}

\medskip
\prr{$\lambda=\pi_4$.}
Let $v_1=\vep_1+\vep_2$, $v_2=\vep_1-\vep_2$, $v_3=\vep_1+\vep_3$,
$v_4=\vep_3$, $v_0=1/2(v_1+v_2)=v_3-v_4$. Obviously, $\{v_0;v_1,\dots,v_4\}$
is an ENSS.

\medskip
Fix $\alpha_0=(1/2,1/2,1/2,1/2)$. We show that $\pi_4$ is contained in the sets~$M(2\pi_1)$,
$M(\pi_2)$, $M(\pi_3)$. Indeed,
$$
\pi_4=\frac12((1,1,2,0)+(1,1,-2,0))\in M(\pi_3).
$$ 
Next,
$s_{\alpha_0}(\pi_2)=(0,-1,-1,-1)$, hence
$\pi_4=\frac12((1,1,1,0)+(1,1,-1,0))\in M(\pi_2)$. Finally 
$s_{\alpha_0}(2\pi_1)=(1,-1,-1,-1)$, which implies
$$
\pi_4=\frac12((1,1,1,-1)+(1,1,-1,1))\in M(2\pi_1).
$$

For all the other dominant weights, their decomposition in the fundamental weights contains either
$2\pi_1$, or $\pi_2$, or $\pi_3$, or $\pi_4$, and applying Lemma~\ref{l.ovl} we deduce the desired conclusion.

\section{The root system~$G_2$}
We realize $G_2$ in the 2-dimensional subspace $\{ \ell_1\vep_1+\ell_2\vep_2+\ell_3\vep_3 \mid \ell_1+\ell_2+\ell_3=0\}$ of the space $\QQ^3$. The following vectors are the roots:
\begin{gather*}
\pm (\vep_1-\vep_2), \; \pm (\vep_1-\vep_3), \; \pm (\vep_2-\vep_3), \; \pm (2\vep_1-\vep_2-\vep_3),
\\
\pm (2\vep_2-\vep_1-\vep_3), \; \pm (2\vep_3-\vep_1-\vep_2).
\end{gather*}

\bigskip
   \begin{figure}[!ht]
   \begin{center}
   \includegraphics[width=2in]{iskluch.1} \\
   Fig.~1
   \end{center}
   \end{figure}
\bigskip

The simple roots are the vectors $\alpha_1=(1, -1, 0)$ and $\alpha_2=(-2,1,1)$.
The fundamental weights are 
$$
  \pi_1=(0,-1,1),\quad
  \pi_2=(-1,-1,2).
$$
It is clear from Fig.~1 that the Weyl group is the dihedral group of order~12. The root and weight lattices coincide.

\slu $\lambda=\pi_1$. \label{sl2}
Then $M(\lambda)=\{\pm (0,1,-1), \pm (1,0,-1), \pm (1,-1,0)\}$, which is hereditarily normal.

\prr $\lambda=\pi_2$. It is the long dominant root, it is clear from Fig.~1 that $M(\lambda)$ coincides with the whole root system.
Consider the following vectors (marked by the thick edges in Fig.~1):
\begin{gather*}
\qmatrix{v_1 \\ v_2 \\ v_3} = \qmatrix{-1 & -1 & 2 \\ -2 & 1 & 1 \\ 0 & -1 & 1}.
\end{gather*}
Then $v_0=1/3(v_1+v_2)=(-1,0,1)=v_1-v_3$. Consider the last coordinate. It is clear that $v_0$ is not a $\ZgO$-combination of vectors 
$v_1$, $v_2$, $v_3$.

All the other dominant weights can be reduced to~$\pi_2$. Indeed, if the vector $\lambda=c_1\pi_1+c_2\pi_2$ has a nonzero coefficient $c_2$, 
then the set $M(\lambda)$ contains $M(\pi_2)$, hence $M(\lambda)$ is not hereditarily normal. If $c_2=0$ but
$\lambda\neq \pi_1$, then $c_2\geqslant 2$. Then the vector $2\pi_1=(0,-2,2)$ belongs to $M(\lambda)$. Consequently, the vectors $(-2,0,2)$ 
and~$(-1,-1,2)=\pi_2$ belong to $M(\lambda)$, see Fig.~1.

\smallskip
Now the proof of Theorem~\ref{mth} is completed.

\end{document}